\documentclass[12pt]{amsart}

\usepackage[T1]{fontenc}
\usepackage[utf8]{inputenc}

\usepackage[margin=1.1in]{geometry}
\usepackage{amsmath, amssymb, amsthm, mathtools}
\usepackage{microtype}
\usepackage{enumitem}
\usepackage{listings}
\usepackage[hidelinks]{hyperref}

\theoremstyle{plain}
\newtheorem{theorem}{Theorem}
\newtheorem{lemma}[theorem]{Lemma}

\theoremstyle{definition}
\newtheorem{remark}[theorem]{Remark}

\lstdefinestyle{python}{
  language=Python,
  basicstyle=\ttfamily\footnotesize,
  keywordstyle=\bfseries,
  commentstyle=\itshape,
  numbers=none,
  breaklines=true,
  breakatwhitespace=true,
  showstringspaces=false,
  frame=single,
  framesep=4pt,
  tabsize=2,
  xleftmargin=14pt,
  columns=fullflexible,
  upquote=true,
}

\title[Mathar's recurrence for OEIS A001711]%
{A short proof of Mathar's 2020 recurrence conjecture for the
 generalized-Stirling sequence A001711}

\author{Tong Niu}


\begin{document}

\maketitle

\begin{abstract}
For the OEIS sequence A001711, contributed by N. J. A. Sloane long
before the on-line era and identified there as the diagonal
$T(n+4, 4)$ of a generalized-Stirling triangle, R. J. Mathar
contributed in February 2020 the conjectured order-2 P-recursive
recurrence
\[
   a(n) - (2n+5)\,a(n-1) + (n+2)^{2}\,a(n-2) \;=\; 0,\qquad n \ge 2.
\]
We give a one-page proof. Detlefs's harmonic-number closed form
$a(n) = \tfrac{1}{4}(n+3)!\,(2 H_{n+3} - 3)$ collapses the
left-hand side, after dividing through by $(n+1)!/4$, to a polynomial
identity of $n$ with coefficient $H_{n+2}$. The harmonic-number
coefficient simplifies to $(n+3) - (2n+5) + (n+2) = 0$ (using
$H_{n+3} = H_{n+2} + \tfrac{1}{n+3}$ and
$H_{n+1} = H_{n+2} - \tfrac{1}{n+2}$); the constant remainder is
$-3 \cdot 0 = 0$ for the same reason. The supplementary archive
contains a SymPy script verifying both pieces symbolically, the
e.g.f.\ expansion against the harmonic closed form, and Mathar's
recurrence numerically for $n = 2, \ldots, 5000$.
\end{abstract}

\section{Introduction}\label{sec:intro}

The On-Line Encyclopedia of Integer Sequences~\cite{OEIS} (henceforth
OEIS) is full of sequences whose ``Conjecture: $\dots$'' comments
record formulas, recurrences, or congruences that were guessed
numerically but never proved. Short rigorous proofs of such
conjectures get published in the \emph{Journal of Integer Sequences},
in \emph{INTEGERS}, in the \emph{Fibonacci Quarterly}, and now and
then in the \emph{Electronic Journal of Combinatorics}.

A lot of these conjectures have been cleared in the last 18 months.
Fried's 2024 and 2025 papers~\cite{Fried2024,Fried2025} closed
several dozen at once; the 2023 list of Kauers and
Koutschan~\cite{KauersKoutschan2023} catalogues the
gold-standard benchmarks for guessed P-recursive recurrences. Most
of the low-hanging fruit in those two sources is gone.

The conjecture treated here is not in those lists. The sequence is
OEIS A001711, recorded as the diagonal $T(n+4, 4)$ of the
generalized-Stirling triangle $\genfrac{[}{]}{0pt}{}{n}{k}_3$
(also written $[n+4, 4]_3$). The first values are
\[
   1,\;7,\;47,\;342,\;2754,\;24552,\;241128,\;2592720,\;
   30334320,\;383970240,\;\ldots
\]
On 16 February 2020, R.~J.~Mathar contributed to A001711 the
conjectured order-2 P-recursive recurrence
\begin{equation}\label{eq:mathar}
   a(n) - (2n+5)\,a(n-1) + (n+2)^{2}\,a(n-2) \;=\; 0,\qquad n \ge 2.
\end{equation}
The conjecture has been sitting open in the OEIS comment thread for
six years; it is not in
\cite{Fried2024,Fried2025,KauersKoutschan2023} or the more recent
Chen-Kauers preprints~\cite{ChenKauers2025}.

\medskip

The proof is one page. Detlefs contributed to A001711 in
2010~\cite{Detlefs2010} the harmonic-number closed form
\begin{equation}\label{eq:closed-form}
   a(n) \;=\; \frac{(n+3)!}{4}\,\bigl(2H_{n+3} - 3\bigr),
\end{equation}
where $H_m = 1 + \tfrac{1}{2} + \dots + \tfrac{1}{m}$ is the $m$-th
harmonic number. We restate \eqref{eq:closed-form} with a
self-contained derivation from the OEIS-listed e.g.f.\ in
Lemma~\ref{lem:closed-form}, then substitute it into
\eqref{eq:mathar} and use $H_{n+3} = H_{n+2} + \tfrac{1}{n+3}$ and
$H_{n+1} = H_{n+2} - \tfrac{1}{n+2}$. Everything collapses to the
polynomial identity $(n+3) - (2n+5) + (n+2) = 0$
(Theorem~\ref{thm:main}). The point is that once the closed form
is in hand, the recurrence is no longer a P-recursive guess but a
two-line consequence of $H_{m+1} - H_m = 1/(m+1)$.

\section{The harmonic-number closed form}\label{sec:closed-form}

\begin{lemma}\label{lem:closed-form}
For every $n \ge 0$ the sequence A001711 satisfies
\[
   a(n) \;=\; \frac{(n+3)!}{4}\,\bigl(2H_{n+3} - 3\bigr).
\]
\end{lemma}

\begin{proof}
Define
\begin{equation}\label{eq:phi}
   \varphi(x) \;=\; \frac{-\log(1-x)}{(1-x)^{3}}.
\end{equation}
Differentiating once,
\begin{equation}\label{eq:phiprime}
   \varphi'(x)
   \;=\; \frac{1}{(1-x)^{4}} + \frac{-3 \log(1-x)}{(1-x)^{4}}
   \;=\; \frac{1 - 3 \log(1-x)}{(1-x)^{4}}.
\end{equation}
The OEIS entry for A001711 lists $\varphi'(x)$ as the exponential
generating function of the sequence: ``E.g.f.: $(d/dx)(-\log(1-x)/(1-x)^{3}) = (1 - 3 \log(1-x))/(1-x)^{4}$.''
So $\sum_{n\ge0} a(n)\, x^{n}/n! = \varphi'(x)$.

Define $b(n) := (n+3)!\,(2H_{n+3} - 3)/4$, which we want to show
equals $a(n)$, and let
\[
   B(x) \;:=\; \sum_{n\ge0} \frac{b(n)}{n!}\,x^{n}
        \;=\; \frac{1}{4}\sum_{n\ge0}(n+1)(n+2)(n+3)\,(2H_{n+3} - 3)\,x^{n}.
\]
Reindexing $m = n + 3$ and recognising $(n+1)(n+2)(n+3)\,x^{n}$ as the
third derivative of $x^{m}$,
\begin{equation}\label{eq:Bx-deriv}
   B(x)
   \;=\; \frac{1}{4}\,\frac{d^{3}}{dx^{3}}\!\sum_{m\ge3}(2H_{m} - 3)\,x^{m}.
\end{equation}
The bracketed sum splits into two known generating functions:
$\sum_{m\ge1} H_{m}\,x^{m} = -\log(1-x)/(1-x)$ and
$\sum_{m\ge0} x^{m} = 1/(1-x)$. Subtracting the $m \in \{1, 2\}$ and
$m \in \{0, 1, 2\}$ tails respectively,
\[
   \sum_{m\ge3}(2H_{m} - 3)\,x^{m}
   \;=\; \frac{-2\log(1-x)}{1-x} - 2x - 3x^{2} - \frac{3 x^{3}}{1-x}.
\]
Apply $d^{3}/dx^{3}$ termwise. Polynomial parts $-2x - 3x^{2}$
contribute $0$. Direct calculation gives
\[
   \frac{d^{3}}{dx^{3}}\!\biggl(\frac{-2\log(1-x)}{1-x}\biggr)
   \;=\; \frac{22 - 12\log(1-x)}{(1-x)^{4}},
\]
and (using $x^{3}/(1-x) = \sum_{m\ge3} x^{m}$, hence
$d^{3}/dx^{3}(x^{3}/(1-x)) = \sum_{k\ge0}(k+1)(k+2)(k+3)\,x^{k} = 6/(1-x)^{4}$)
\[
   \frac{d^{3}}{dx^{3}}\!\biggl(\frac{-3x^{3}}{1-x}\biggr)
   \;=\; \frac{-18}{(1-x)^{4}}.
\]
Add the two and divide by $4$:
\[
   B(x)
   \;=\; \frac{1}{4}\!\cdot\!\frac{22 - 18 - 12\log(1-x)}{(1-x)^{4}}
   \;=\; \frac{1 - 3\log(1-x)}{(1-x)^{4}}
   \;=\; \varphi'(x).
\]
So $\sum b(n)\,x^{n}/n! = \sum a(n)\,x^{n}/n!$, and
$b(n) = a(n)$ for every $n \ge 0$.
\end{proof}

\begin{remark}\label{rem:closed-form-cite}
The closed form \eqref{eq:closed-form} was contributed to the
A001711 page by W. Edwin Clark on a comment from R. R. Detlefs in
2010~\cite{Detlefs2010}. We restate it here with a self-contained
derivation so the proof of Theorem~\ref{thm:main} below is
self-contained as well.
\end{remark}

\section{Proof of Mathar's recurrence}\label{sec:proof}

\begin{theorem}\label{thm:main}
The sequence $a(n)$ of OEIS A001711 satisfies Mathar's
recurrence~\eqref{eq:mathar}: for every $n \ge 2$,
\[
   a(n) - (2n+5)\,a(n-1) + (n+2)^{2}\,a(n-2) \;=\; 0.
\]
\end{theorem}

\begin{proof}
Substitute Lemma~\ref{lem:closed-form} into the left-hand side and
factor $(n+1)!/4$ out of every term. With the abbreviation
$h_{m} := 2H_{m} - 3$,
\begin{equation}\label{eq:M-over}
   \frac{4\,M(n)}{(n+1)!}
   \;=\; (n+2)(n+3)\,h_{n+3}
       \,-\, (2n+5)(n+2)\,h_{n+2}
       \,+\, (n+2)^{2}\,h_{n+1}.
\end{equation}
Pull out the common factor $(n+2)$:
\begin{equation}\label{eq:M-bracket}
   \frac{4\,M(n)}{(n+1)!\,(n+2)}
   \;=\; (n+3)\,h_{n+3}
       \,-\, (2n+5)\,h_{n+2}
       \,+\, (n+2)\,h_{n+1}.
\end{equation}
Now use the elementary $H_{n+3} = H_{n+2} + \tfrac{1}{n+3}$ and
$H_{n+1} = H_{n+2} - \tfrac{1}{n+2}$, so
$h_{n+3} = h_{n+2} + \tfrac{2}{n+3}$ and
$h_{n+1} = h_{n+2} - \tfrac{2}{n+2}$. Substitute into
\eqref{eq:M-bracket} and collect the $h_{n+2}$-multiple and the
remainder:
\begin{align*}
   (n+3)\,h_{n+3}
   &\;=\; (n+3)\,h_{n+2} + 2,\\
   (2n+5)\,h_{n+2}
   &\;=\; (2n+5)\,h_{n+2},\\
   (n+2)\,h_{n+1}
   &\;=\; (n+2)\,h_{n+2} - 2.
\end{align*}
Adding with signs $+, -, +$:
\begin{align*}
   \frac{4\,M(n)}{(n+1)!\,(n+2)}
   &\;=\; \bigl[(n+3) - (2n+5) + (n+2)\bigr]\,h_{n+2}
       \;+\;\bigl[2 - 0 - 2\bigr]\\
   &\;=\; 0\cdot h_{n+2} + 0
   \;=\; 0.
\end{align*}
So $M(n) = 0$ for every $n \ge 2$ (where the prefactor
$(n+1)!\,(n+2)/4$ is nonzero), which is Mathar's recurrence.
\end{proof}

\section{Remarks}\label{sec:remarks}

The same Mathar conjecture cleanups for OEIS A002627, A025166,
A176677, and A214615 in the present
series~\cite{Niu2026d-finite-2,Niu2026sequence-3,Niu2026sequence-4,Niu2026sequence-5}
run the proof through a polynomial-coefficient ODE on the EGF, which
spits out the recurrence after coefficient extraction. Here we get
away with less. The harmonic-number closed form \eqref{eq:closed-form}
already exists on the OEIS page; the elementary recurrence
$H_{m+1} - H_m = 1/(m+1)$ is enough to finish the job, with no
ODE in sight.

There is, of course, an EGF-ODE route too. Write
$(1-x)\varphi' - 4\varphi = 3/(1-x)^{4}$, differentiate, eliminate
the $1/(1-x)^{5}$ remainder, multiply through by $(1-x)$. Out comes
the second-order linear ODE
$(1-x)^{2}\,\varphi'' - 9(1-x)\,\varphi' + 16\,\varphi = 0$, and
Mathar's recurrence drops out from the standard EGF coefficient
identities. We omit the calculation.

The supplementary script \texttt{verify\_proof.py}
(Appendix~\ref{app:verifier}) checks the closed form
\eqref{eq:closed-form} against the OEIS b-file for the first $20$
values; expands the e.g.f.\ \eqref{eq:phiprime} as a power series and
matches it against the harmonic closed form (for $n = 0, \ldots, 9$);
performs the symbolic reduction of Mathar's left-hand side to
$0 \cdot H_{n+2} + 0$ via SymPy (this is the proof's key step); and
verifies Mathar's recurrence \eqref{eq:mathar} numerically for every
$n$ from $2$ up to $5000$.

\section{Acknowledgments}

The author declares no competing interests.

AI-assisted tools were used in the preparation of this manuscript,
including for drafting proof outlines and generating the symbolic
verification code. The author verified all mathematical claims
independently and takes full responsibility for the results.

\appendix

\section{Verifier source (machine-checkable)}\label{app:verifier}

The script referenced in \S\ref{sec:remarks} is reproduced in full
below. It depends only on SymPy; no Maple, Mathematica, or any other
external CAS license is required.

\subsection*{verify\_proof.py (closed form, e.g.f., symbolic reduction, numeric Mathar)}
\begin{lstlisting}[style=python]
"""Verify the proof of Mathar's 2020 conjectured recurrence for OEIS
A001711, the diagonal a(n) = T(n+4, 4) of a generalized Stirling
triangle (with offset 0, also written [n+4, 4]_3).

Proof outline (see the paper for details):

1. (Detlefs, 2010, restated.) The sequence has the closed form
       a(n) = ((n+3)! / 4) * (2 H_{n+3} - 3),
   where H_m = 1 + 1/2 + ... + 1/m. This follows from the e.g.f.
   F(x) = (1 - 3 log(1-x)) / (1-x)^4 listed on the OEIS page; see
   the paper for a self-contained derivation.

2. Substitute the closed form into Mathar's recurrence
       a(n) - (2n+5) a(n-1) + (n+2)^2 a(n-2) = 0,  n >= 2,
   factor out (n+1)! / 4, and reduce H_{n+3}, H_{n+1} to H_{n+2}
   via the elementary identities H_{k+1} - H_k = 1 / (k+1). The
   result is the polynomial identity
       (n+3) - (2n+5) + (n+2) = 0,
   which holds for all n.

This script checks both the closed form (against the OEIS b-file) and
Mathar's recurrence (numerically up to n = 5000 in exact rational
arithmetic), and verifies the symbolic identity above.
"""
from __future__ import annotations

from fractions import Fraction
from math import factorial

import sympy as sp


# ---------------------------------------------------------------------------
# closed form for a(n)
# ---------------------------------------------------------------------------

def harmonic(n: int) -> Fraction:
    """Exact rational n-th harmonic number H_n = 1 + 1/2 + ... + 1/n."""
    return sum((Fraction(1, k) for k in range(1, n + 1)), Fraction(0))


def a_closed_form(n: int) -> int:
    """a(n) = ((n+3)! / 4) * (2 H_{n+3} - 3)."""
    val = Fraction(factorial(n + 3), 4) * (2 * harmonic(n + 3) - 3)
    assert val.denominator == 1, n
    return val.numerator


# OEIS b-file values (first 20 entries; full b001711.txt available at
# https://oeis.org/A001711/b001711.txt).
OEIS_REFERENCE = [
    1, 7, 47, 342, 2754, 24552, 241128, 2592720, 30334320, 383970240,
    5231113920, 76349105280, 1188825724800, 19675048780800,
    344937224217600, 6386713749964800, 124548748102195200,
    2551797512248320000, 54804198761303040000, 1231237843834521600000,
]


def check_closed_form_matches_oeis() -> None:
    """The harmonic-number closed form reproduces the OEIS reference values."""
    for n, expected in enumerate(OEIS_REFERENCE):
        got = a_closed_form(n)
        if got != expected:
            print(f"FAIL closed-form: a({n}) = {got}, expected {expected}")
            return
    print(f"PASS closed-form vs OEIS: {len(OEIS_REFERENCE)} values match")


# ---------------------------------------------------------------------------
# Mathar's conjectured recurrence
# ---------------------------------------------------------------------------

def mathar_LHS_rational(n: int) -> Fraction:
    """LHS of Mathar's conjectured recurrence at index n, computed exactly
    in Fraction arithmetic via the harmonic-number closed form."""
    A = Fraction(factorial(n + 3), 4) * (2 * harmonic(n + 3) - 3)
    B = Fraction(factorial(n + 2), 4) * (2 * harmonic(n + 2) - 3)
    C = Fraction(factorial(n + 1), 4) * (2 * harmonic(n + 1) - 3)
    return A - (2 * n + 5) * B + (n + 2) ** 2 * C


def check_mathar_numerically(N: int = 5000) -> None:
    """Verify Mathar's recurrence holds for every n in 2..N."""
    fails = 0
    for n in range(2, N + 1):
        if mathar_LHS_rational(n) != 0:
            fails += 1
            if fails <= 3:
                print(f"FAIL mathar at n={n}")
    if fails == 0:
        print(f"PASS Mathar recurrence numerically: n = 2..{N} ({N - 1} values)")
    else:
        print(f"FAIL Mathar recurrence: {fails} failures in n = 2..{N}")


# ---------------------------------------------------------------------------
# symbolic reduction of Mathar's recurrence to a 3-term polynomial identity
# ---------------------------------------------------------------------------

def check_mathar_symbolically() -> None:
    """The proof's key step is: substitute the closed form into Mathar's
    recurrence, factor out (n+1)!/4, and use H_{n+3} = H_{n+2} + 1/(n+3)
    plus H_{n+1} = H_{n+2} - 1/(n+2). The resulting expression is a sum
    of an H_{n+2}-multiplier and an H-free remainder; both have to be 0
    as polynomials in n. SymPy verifies both.
    """
    n, H = sp.symbols("n H", real=True)

    # Substitute the closed form a(k) = (k+3)!/4 * (2 H_{k+3} - 3) and
    # divide through by (n+1)!/4. The factorial ratio (n+3)!/(n+1)! is
    # (n+2)(n+3); (n+2)!/(n+1)! is (n+2); (n+1)!/(n+1)! is 1.
    H_n_plus_3 = H + 1 / (n + 2) + 1 / (n + 3)  # H_{n+3} = H_{n+2} + 1/(n+3); but we want everything in H_{n+2}, so wait...
    # Convention: let H denote H_{n+2}. Then
    #   H_{n+3} = H + 1/(n+3),
    #   H_{n+2} = H,
    #   H_{n+1} = H - 1/(n+2).
    H_n_plus_3 = H + 1 / (n + 3)
    H_n_plus_2 = H
    H_n_plus_1 = H - 1 / (n + 2)

    # Closed-form values as multiples of (n+1)!/4:
    A_over = (n + 2) * (n + 3) * (2 * H_n_plus_3 - 3)
    B_over = (n + 2) * (2 * H_n_plus_2 - 3)
    C_over = (2 * H_n_plus_1 - 3)

    M_over = A_over - (2 * n + 5) * B_over + (n + 2) ** 2 * C_over
    expanded = sp.expand(M_over)

    # Extract H-coefficient and the H-independent remainder.
    poly = sp.Poly(expanded, H)
    coeff_H1 = poly.nth(1)
    coeff_H0 = poly.nth(0)
    coeff_H1_simp = sp.simplify(coeff_H1)
    coeff_H0_simp = sp.simplify(coeff_H0)
    if coeff_H1_simp == 0 and coeff_H0_simp == 0:
        print("PASS Mathar identity reduces to (n+3) - (2n+5) + (n+2) = 0 after H_{n+2}-substitution")
    else:
        print(f"FAIL: H coefficient = {coeff_H1_simp}, constant remainder = {coeff_H0_simp}")


# ---------------------------------------------------------------------------
# Sanity check on the e.g.f. closed form (alternative derivation route).
# ---------------------------------------------------------------------------

def check_egf_first_terms(N: int = 10) -> None:
    """The OEIS-listed e.g.f. F(x) = (1 - 3 log(1-x)) / (1-x)^4 should
    expand with [x^n] / (1/n!) = a(n) for the first N indices."""
    x = sp.symbols("x")
    F = (1 - 3 * sp.log(1 - x)) / (1 - x) ** 4
    series = sp.series(F, x, 0, N + 1).removeO()
    fails = 0
    for k in range(N):
        coeff = sp.simplify(series.coeff(x, k))
        a_from_egf = sp.simplify(coeff * sp.factorial(k))
        a_from_harmonic = a_closed_form(k)
        if int(a_from_egf) != a_from_harmonic:
            fails += 1
            print(f"FAIL e.g.f. coefficient: a({k}) = {a_from_egf} via EGF, {a_from_harmonic} via harmonic")
    if fails == 0:
        print(f"PASS e.g.f. F(x) = (1 - 3 log(1-x)) / (1-x)^4 matches harmonic closed form for n = 0..{N-1}")


if __name__ == "__main__":
    check_closed_form_matches_oeis()
    check_egf_first_terms(N=10)
    check_mathar_symbolically()
    check_mathar_numerically(N=5000)
\end{lstlisting}

\end{document}